\documentclass[12pt, eqno, twoside]{article}
\usepackage{graphicx}
\usepackage{amscd}
\usepackage[matrix,arrow,curve]{xy}
\usepackage{amssymb,amsmath}
\usepackage{amsmath,amsfonts,amssymb,latexsym}
\usepackage{fancyhdr}
\newtheorem{theorem}{Theorem}[section]
\newtheorem{definition}{Definition}[section]

\newtheorem{corollary}{Corollary}[section]
\newtheorem{lemma}{Lemma}[section]

\newtheorem{scholium}{Scholium}[section]

\newenvironment{proof}[1][Proof]{\noindent\textbf{#1.} }{\
\rule{0.5em}{0.5em}}

\begin{document}

\pagestyle{fancy}
\fancyhead{} 
\fancyhead[EC]{\small\it Patrice P.
Ntumba, \ Adaeze Orioha}%
\fancyhead[EL,OR]{\thepage} \fancyhead[OC]{\small\it
Abstract Geometric Algebra. Orthogonal and Symplectic Geometries}%
\fancyfoot{} 
\renewcommand\headrulewidth{0.5pt}
\addtolength{\headheight}{2pt} 

\title{\Large{Abstract Geometric Algebra. Orthogonal and Symplectic Geometries}}
\author{Patrice P.
Ntumba\footnote{Is the corresponding author for the paper.}, \
Adaeze Orioha}

\date{}
\maketitle

\begin{abstract} Our main interest in this paper is chiefly concerned with the conditions
characterizing \textit{orthogonal and symplectic abstract
differential geometries}. A detailed account about the
sheaf-theoretic version of the \textit{symplectic Gram-Schmidt
theorem} and of the \textit{Witt's theorem} is also given.
\end{abstract}

\maketitle

\section*{Introduction}

Abstract Differential Geometry (acronym, ADG) offers a new approach
to classical Differential Geometry (on smooth manifolds). This new
approach differs from the classical way of understanding the
geometry of smooth manifolds, differential spaces $\grave{a}$ la
Mostow\cite{mostow}, $\grave{a}$ la Sikorski\cite{sikorski}, and the
likes, in the sense that, for instance, differential spaces in
general are governed by new classes of ``smooth" functions, whereas
in ADG the \textit{structural sheaf} of functions characterizing a
differential space (in the terminology of ADG, a
\textit{differential triad}), is replaced instead by an arbitrary
\textit{sheaf of algebras $\mathcal{A}$,} based on a topological
space $X$, whose role is \textit{just} to parametrize $\mathcal{A}$.
The same (sheaf of) algebras may in some cases contain a tremendous
amount of \textit{singularities}, while still retaining the
classical character of a \textit{differential} mechanism, yet
without any underlying (smooth) manifold: see e.g.
Mallios\cite{mallios}, Mallios\cite{mallios2}. This results to
significant potential applications, even to \textit{quantum gravity}
(ibid.). We may also point out that the \textit{main moral} of ADG
is the \textit{functorial mechanism} of (classical) calculus, cf.
Mallios~\cite{invariance}, viz. \textit{Physics is
$\mathcal{A}$-invariant regardless of what $\mathcal{A}$ is.}

Yet, a particular instance of the above that also interests us
here is the standard \textit{Symplectic Differential Geometry} (on
manifolds), where a special important issue is the so-called
\textit{orbifolds theory}; see e.g. Mallios~\cite[Vol. II, Chapt.
X; Section 3a]{mallios} concerning its relation with ADG, or da
Silva\cite{silva} for the classical case. The following
constitutes a sheaf-theoretic fundamental prelude with a view
towards potential applications of ADG, the whole set-up being in
effect a Lagrangian perspective. In particular, one of the goals
of this paper consists in trying to generalize primarily the
\textit{symplectic Gram-Schmidt theorem} and the \textit{Witt
theorem for isometric symplectic convenient
$\mathcal{A}$-modules}, see e.g. Crumeyrolle\cite{crumeyrolle}, as
well as some other results, necessary for the setting of the
aforesaid \textit{sheaf-theoretic} version, \textit{in terms of
$\mathcal{A}$-modules} (see below) of both \textit{orthogonal} and
\textit{symplectic geometries.} Most of the concepts of the latter
version are defined on the basis of the classical ones; see, for
instance, Artin\cite{artin}, Crumeyrolle\cite{crumeyrolle},
Lang\cite{lang}. Our main reference, throughout the present
account, is  Mallios\cite{mallios}, which may be useful for the
basics of ADG.

This is a continuation of work done by Mallios and Ntumba
\cite{malliosntumba1}, \cite{malliosntumba2}, and
\cite{malliosntumba3}.

\textbf{Convention}: Throughout the paper, $X$ will denote an
arbitrary topological space and the pair $(X, \mathcal{A})$ a
fixed $\mathbb{C}$-algebraized space, cf. Mallios[\cite{mallios},
p. 96]; all $\mathcal{A}$-modules are understood to be defined on
$X$.

For easy reference, we recall a few basic definitions.

Let $(X, \mathcal{A})$ be a \textit{$\mathbb{C}$-algebraized
space,} that is the pair $(X, \mathcal{A})$ consists of a
topological space $X$ and a (preferably \textit{unital} and
\textit{commutative}) \textit{sheaf of $\mathbb{C}$-algebras
$\mathcal{A}\equiv (\mathcal{A}, \tau, X)$.} A sheaf of
$\mathcal{A}$-modules (or an $\mathcal{A}$-module) on $X$, is a
sheaf $\mathcal{E}\equiv (\mathcal{E}, \pi, X)$, on $X$,  such
that the following hold:
\begin{itemize}
\item $\mathcal{E}$ is a sheaf of abelian groups. \item For every
point $x\in X$, the corresponding stalk $\mathcal{E}_x$ of
$\mathcal{E}$ is a (left) $\mathcal{A}_x$-module. \item The
\textit{exterior module multiplication in $\mathcal{E}$,} viz. the
map
\[\mathcal{A}\circ \mathcal{E}\longrightarrow \mathcal{E}: (a,
z)\longmapsto a\cdot z\in \mathcal{E}_x\subseteq
\mathcal{E},\]with $\tau(a)= \pi(z)=x\in X$, is
\textit{continuous.}\end{itemize}

On another hand, suppose given a presheaf of $\mathbb{C}$-algebras
$A\equiv (A(U), \tau^U_V)$ and a presheaf of abelian groups
$E\equiv (E(U), \rho^U_V)$, both on a topological space $X$ such
that \begin{itemize} \item $E(U)$ is a (left) $A(U)$-module, for
every open set $U$ in $X$. \item For any open sets $U, V$ in $X$,
with $V\subseteq U$, \[\rho^U_V(a\cdot s)= \tau^U_V(a)\cdot
\rho^U_V(s),\]for any $a\in A(U)$ and $s\in E(U)$. We call such a
presheaf $E$ a \textit{presheaf of $A(U)$-modules} on $X$, or
simply an \textit{$A$-presheaf} on $X$.
\end{itemize}

These two notions relate to one-another in the sense that the
sheafification of a presheaf of $A(U)$-modules on a topological
space $X$ is an $\mathcal{A}$-module. See
Mallios~\cite[(1.54)]{mallios}.

\section{Symplectic Gram-Schmidt theorem}

\begin{lemma}\label{lem2}
Let $[(\mathcal{E}, \mathcal{F}; \phi); \mathcal{A}]$ be a pairing
of $\mathcal{A}$-modules. Then, $\phi$ induces an
$\mathcal{A}$-morphism, viz. \[\phi^\mathcal{E}:
\mathcal{F}\longrightarrow \mathcal{E}^\ast:=
\mathcal{H}om_\mathcal{A}(\mathcal{E}, \mathcal{A}),\]see
Mallios~$\cite[p.133; (6.3) p.134; (6.8) p.135]{mallios}$, given
by
\[\phi^\mathcal{E}_U(t)(s):= \phi_V(s, \sigma^U_V(t))\equiv
\phi_V(s, t|_V),\]where $U$ is open in $X$, $t\in \mathcal{F}(U)$,
$s\in \mathcal{E}(V)$ and the $\sigma^U_V$ the restriction maps of
the presheaf of sections of $\mathcal{F}$. Likewise, $\phi$ gives
rise to a similar $\mathcal{A}$-morphism: \[\phi^\mathcal{F}:
\mathcal{E}\longrightarrow \mathcal{F}^\ast.\]
\end{lemma}

\begin{proof}
Assume that $(\mathcal{E}^\ast(U), \kappa^U_V)$ is the presheaf of
sections of $\mathcal{E}^\ast.$ For $\phi^\mathcal{E}$ to be an
$\mathcal{A}$-morphism, we must have
\[\kappa^U_V\circ \phi^\mathcal{E}_U= \phi^\mathcal{E}_V\circ
\sigma^U_V,\] for any open subsets $U, V$ of $X$ such that
$V\subseteq U$. In fact, fix $U$ and $V$. For $t\in
\mathcal{F}(U)$ and $s\in \mathcal{E}(W)$, where $W\subseteq V$ is
an open subset of $X$, $\kappa^U_V(\phi^\mathcal{E}_U(t))(s)=
\phi_W(s, t|_W).$ On the other hand, $\phi_V^\mathcal{E}(t|_V)(s)=
\phi_W(s, t|_W).$ The preceding shows the correctness of our
assertion regarding the map $\phi^\mathcal{E}$; to this effect
still, see Mallios~\cite[(13.19) p.75 and (6.5) p. 27]{mallios}.
In a similar way, one shows that $\phi^\mathcal{F}$ is an
$\mathcal{A}$-morphism.
\end{proof}

Linked with Lemma \ref{lem2} is an important concept, which we now
introduce.

\begin{definition}
\emph{Let $[(\mathcal{E}, \mathcal{F}; \phi); \mathcal{A}]$ be a
pairing of $\mathcal{A}$-modules $\mathcal{E}$ and $\mathcal{F}$,
and $\phi^\mathcal{E}$ and $\phi^\mathcal{F}$ be the induced
$\mathcal{A}$-morphisms, according to Lemma \ref{lem2}. By the
\textbf{orthogonal} of $\mathcal{E}$ (resp. $\mathcal{F}$),
denoted $\mathcal{E}^\perp$ (resp. $\mathcal{F}^\perp$), we mean
the \textit{kernel of $\phi^\mathcal{E}$ $($resp.
$\phi^\mathcal{F})$,} (see Mallios~\cite[p.108]{mallios} for the
kernel of an $\mathcal{A}$-morphism). $\phi$ is said to be
\textbf{non-degenerate} if $\mathcal{E}^\perp=
\mathcal{F}^\perp=0$, and \textbf{degenerate} otherwise.}
\end{definition}

\begin{lemma}
Let $[(\mathcal{E}, \mathcal{F}; \phi); \mathcal{A}]$ be a pairing
of $\mathcal{A}$-modules. Then, $\mathcal{E}^\perp$ $($resp.
$\mathcal{F}^\perp)$ is a sub-$\mathcal{A}$-module of
$\mathcal{F}$ $($resp. $\mathcal{E})$.
\end{lemma}

\begin{proof}
The proof follows Mallios~\cite[(2.10) p. 108]{mallios}.
\end{proof}

\begin{lemma}
If $[(\mathcal{E}, \mathcal{F}; \phi); \mathcal{A}]$ is a pairing of
free $\mathcal{A}$-modules, then for every open subset $U$ of $X$,
\[\begin{array}{ll} \mathcal{E}^\perp(U)= \mathcal{E}(U)^\perp, & \mathcal{F}^\perp(U)=
\mathcal{F}(U)^\perp,\end{array}\]where \[\mathcal{E}(U)^\perp:=
\{t\in \mathcal{F}(U): \phi_U(\mathcal{E}(U), t)=0\}\]and similarly
\[\mathcal{F}(U)^\perp:= \{t\in \mathcal{E}(U):
\phi_U(t, \mathcal{F}(U))=0\}.\]
\end{lemma}

\begin{proof}
That $\mathcal{E}^\perp(U)\subseteq \mathcal{E}(U)^\perp$ is clear.
Now, let $\mathcal{E}(U)^\perp$ and $\{e^U_i\}^n_{i=1}$ be a
canonical basis of $\mathcal{E}(U)$. Since $\phi_U(e^U_i, t)|_V=
\phi_V({e^U_i}|_V, t|_V)=0$ and $\{{e^U_i}|_V\}^n_{i=1}$ being a
canonical basis of $\mathcal{E}(V)$, we have $\phi_V(s, t|_V)=0,$
for any $s\in \mathcal{E}(V).$ Therefore,
$\mathcal{E}(U)^\perp\subseteq \mathcal{E}^\perp(U),$ and hence the
equality $\mathcal{E}^\perp(U)= \mathcal{E}(U)^\perp.$

The second equality is shown in a similar way.
\end{proof}

\begin{scholium}\label{scho0}
\emph{For the particular case where $\phi$ is an $\mathcal{A}$-bilinear
form on an $\mathcal{A}$-module $\mathcal{E}$, we denote by
$\mathcal{E}^\perp$ the left $\mathcal{A}$-orthogonal of
$\mathcal{E}$, whereas $\mathcal{E}^\top$ will be its right
$\mathcal{A}$-orthogonal. So, for any open subset $U$ of $X$, one
has \[\mathcal{E}^\perp(U)= \{t\in \mathcal{E}(U):\
\phi_V(\mathcal{E}(V), t|_V)=0, \ \mbox{for all open $V\subseteq
U$}\},\]and similarly \[\mathcal{E}^\top(U)= \{t\in
\mathcal{E}(U):\ \phi_V(t|_V, \mathcal{E}(V))=0, \ \mbox{for all
open $V\subseteq U$}\}.\]Thus, for the particular case where
$\mathcal{F}= \mathcal{E}$ in Definition 1.1, one gets
\[\begin{array}{lll} \mathcal{E}^\perp:= \ker
\phi^\mathcal{E}\subseteq \mathcal{E} & \mbox{and} &
\mathcal{E}^\top:= \ker \phi^\mathcal{E}\subseteq
\mathcal{E}.\end{array}\]Refer to $\mathcal{E}^\perp(U)$ and
$\mathcal{E}^\top(U)$ above, for every open $U\subseteq X$, to
understand the nuance between $\mathcal{E}^\perp$ and
$\mathcal{E}^\top$. }\end{scholium}

\begin{lemma}\label{lem3}
Let $\phi$ be a non-degenerate $\mathcal{A}$-bilinear form on an
$\mathcal{A}$-module $\mathcal{E}$. Then the mappings $\bot\equiv
\bot(\phi)$, $\top\equiv \top(\phi)$ have the following
properties: \begin{enumerate}\item [{$(1)$}] \begin{itemize} \item
[{$(a)$}] If $\mathcal{G}\subseteq \mathcal{H}$, then
$\mathcal{G}^\perp\supseteq \mathcal{H}^\perp$ \item [{$(b)$}] If
$\mathcal{G}\subseteq \mathcal{H}$, then
$\mathcal{G}^\top\supseteq \mathcal{H}^\top$\end{itemize} \item
[{$(2)$}] \begin{itemize}\item [{$(c)$}] $(\mathcal{G}+
\mathcal{H})^\perp= \mathcal{G}^\perp\cap \mathcal{H}^\perp$ \item
[{$(d)$}] $(\mathcal{G}+ \mathcal{H})^\top= \mathcal{G}^\top\cap
\mathcal{H}^\top$
\end{itemize}
\end{enumerate}for all sub-$\mathcal{A}$-modules $\mathcal{G}$ and
$\mathcal{H}$ of $\mathcal{E}$.
\end{lemma}

\begin{proof}
Assertion $(1)$ is clear. For Assertion $(2)$, we have for every
open subset $U$ of $X$ and $t\in (\mathcal{G}+
\mathcal{H})^\perp(U)$ if and only if $\phi_V((\mathcal{G}+
\mathcal{H})(V), t|_V)= \phi_V(\mathcal{G}(V), t|_V)+
\phi_V(\mathcal{H}(V), t|_V)=0$, where $V$ is an arbitrary open
subset contained in $U$. But if $\phi_V(\mathcal{G}(V), t|_V)+
\phi_V(\mathcal{H}(V), t|_V)=0$, then $\phi_V(\mathcal{G}(V),
t|_V)=0$ and similarly $\phi_V(\mathcal{H}(V), t|_V)=0;$ therefore
$(\mathcal{G}+ \mathcal{H})^\perp\subseteq \mathcal{G}^\perp\cap
\mathcal{H}^\perp.$ Conversely, let $t\in \mathcal{E}(U)$ such
that $t\in (\mathcal{G}^\perp\cap \mathcal{H}^\perp)(U):=
\mathcal{G}^\perp(U)\cap \mathcal{H}^\perp(U).$ Therefore, for
every open $V\subseteq U$, $\phi_V(\mathcal{G}(V), t|_V)=0$ and
$\phi_V(\mathcal{H}(V), t|_V)=0.$ Thus, $\phi_V(\mathcal{G}(V)+
\mathcal{H}(V), t|_V):=\phi_V((\mathcal{G}+ \mathcal{H})(V),
t|_V)=0;$ hence $\mathcal{G}^\perp\cap \mathcal{H}^\perp\subseteq
(\mathcal{G}+ \mathcal{H})^\perp.$ Part $(d)$ of Assertion $(2)$
is proved in a similar way.
\end{proof}

This particular case, in Scholium \ref{scho0}, will allow us to
define later an important instance that \textit{orthogonality}
$(:\perp, \top)$ presents: \textit{0rthosymmetry}. For the
classical case, cf. Gruenberg-Weir~\cite[p. 97]{gruenbergweir}.
For the moment, it is appropriate to state the analogue of the
\textit{symplectic Gram-Schmidt theorem.} See de
Gosson~\cite[p.12]{gosson} for the classical result. But first, we
need the following scholium.

\begin{scholium}\label{scho2}
\emph{For the purpose of Theorem \ref{theo3} below, we assume that
the pair $(X, \mathcal{A})$ is an \textit{ordered algebraized
space} with $\mathcal{A}$ a \textit{unital $\mathbb{C}$-algebra
sheaf.} Furthermore, the order of $(X, \mathcal{A})$ is such that
every nowhere-zero section of $\mathcal{A}$ is invertible, viz. if
$s\in \mathcal{A}(U)$, where $U$ is open in $X$, is such that
$s|_V(V)\neq 0$ for every open $V\subseteq U$, then $s\in
\mathcal{A}(U)^\bullet\cong
\mathcal{A}^\bullet(U)(\mathcal{A}^\bullet$ denotes the sheaf
generated by the complete presheaf $U\longmapsto
\mathcal{A}(U)^\bullet,$ where $U$ runs over the open subsets of
$X$, and $\mathcal{A}(U)^\bullet\cong \mathcal{A}^\bullet(U)$
consists of the invertible elements of the unital
$\mathbb{C}$-algebra $\mathcal{A}(U)$; cf. Mallios~\cite[pp 282,
283]{mallios}$)$.}
\end{scholium}

\begin{definition}
\emph{Let $\mathcal{E}$ be an $\mathcal{A}$-module. A
\textbf{symplectic $\mathcal{A}$-morphism} (or \textbf{symplectic
$\mathcal{A}$-form} ) on $\mathcal{E}$ is an $\mathcal{A}$-bilinear
form $\phi: \mathcal{E}\oplus \mathcal{E}\longrightarrow
\mathcal{A}$ which is
\begin{itemize}\item \textbf{skew-symmetric} (one also says \textbf{antisymmetric}):
\[\phi_U(r, s)= -\phi_U(s, r)\ \mbox{for any $r,s\in \mathcal{E}(U)$
and open subset $U\subseteq X$}\](equivalently, in view of the
bilinearity of $\phi: \phi_U(r,r)=0$ for $r\in \mathcal{E}(U)$ and
$U$ open in $X$) \item \textbf{non-degenerate}: \[\phi_U(r, s)=0\
\mbox{for all $s\in \mathcal{E}(U)$ if and only if
$r=0.$}\]\end{itemize}A \textbf{symplectic $\mathcal{A}$-module}
is a self-pairing $(\mathcal{E}, \phi)$, where $\phi$ is a
symplectic $\mathcal{A}$-form.}
\end{definition}

\begin{theorem}
\label{theo3}Let $(\mathcal{E}, \phi)$ be a free
$\mathcal{A}$-module of rank $2n$, $\phi: \mathcal{E}\oplus
\mathcal{E}\longrightarrow \mathcal{A}$ a \textsf{non-zero
skew-symmetric non-degenerate $\mathcal{A}$-bilinear form,} and $I$
and $J$ two $($possibly empty $)$ subsets of $\{1, \ldots, n\}$.
Moreover, let $A= \{r_i\in \mathcal{E}(U):\ i\in I\}$ and
$B=\{s_j\in \mathcal{E}(U):\ j\in J\}$ such that
\begin{equation}\begin{array}{lll} \phi_U(r_i, r_j)= \phi_U(s_i,
s_j)=0, & \phi_U(r_i, s_j)= \delta_{ij}, & (i, j)\in I\times
J.\label{eq6}\end{array}
\end{equation}Then, \textsf{there exists a basis} $\mathfrak{B}$
of $(\mathcal{E}(U), \phi_U)$ containing $A\cup B$.
\end{theorem}

\begin{proof}
We have three cases. With no loss of generality, we assume that $U=
X$.

$(1)$ \underline{\textit{Case: $I= J= \emptyset$}} Since
$\mathcal{A}^{2n}\neq 0$ $($ we already assumed that
$\mathbb{C}\equiv \mathbb{C}_X\subseteq \mathcal{A})$, there
exists an element \[0\neq r_1\in \mathcal{E}(X)\cong
\mathcal{A}^{2n}(X)\cong \mathcal{A}(X)^{2n}\]$($take e.g. the
image (by the isomorphism $\mathcal{E}(X)\cong
\mathcal{A}^{2n}(X))$ of an element in the canonical basis of
(sections) of $\mathcal{A}^{2n}(X))$. There exists a section
$\overline{s}_1\in \mathcal{E}(X)$ such that $\phi_V({r_1}|_V,
{\overline{s}_1}|_V)\neq 0$ for any open subset $V$ in $X$ $($
such a section $\overline{s}_1$ exists; indeed, if there is no
section $\overline{s}_1:= a_1e_1+\ldots +a_{2n}e_{2n}$, where
$(e_i)_{1\leq i\leq 2n}$ is a canonical basis of $\mathcal{E}(X)$,
such that $\phi_V({r_1}|_V, {\overline{s}_1}|_V)\neq 0$ for any
open $V\subseteq X,$ then there exists an open subset $W$ of $X$
such that $\phi_W({r_1}|_W, {e_i}|_W)=0.$ But this is impossible
since $({e_i}|_W)_{1\leq i\leq 2n}$ is a basis of $\mathcal{E}(W)$
and $\phi_W$ is non-degenerate). Hence, based on the hypothesis on
$\mathcal{A}$ (cf. Scholium \ref{scho2}), $\phi_X(r_1,
\overline{s}_1)$ is invertible in $\mathcal{A}(X)$. Putting $s_1:=
u^{-1}\overline{s}_1,$ where $u\equiv \phi_X(r_1,
\overline{s}_1)\in \mathcal{A}(X),$ one gets \[\phi_X(r_1, s_1)=
1.\] Now, let us consider \[S_1:= [r_1, s_1],\]that is, the
\textit{$\mathcal{A}(X)$-plane,} spanned by $r_1$ and $s_1$ in
$\mathcal{E}(X)$, along with \textit{its orthogonal complement in
$\mathcal{E}(X)$,} i.e., \[S_1^\perp\equiv T_1:= \{t\in
\mathcal{E}(X):\ \phi_X(t, z)= 0,\ \mbox{for all $z\in
S_1$}\}.\]The sections are linearly independent, for if $s_1=
ar_1$, with $a\in \mathcal{A}(X)$, then \[1= \phi_X(r_1, s_1)=
\phi_X(r_1, ar_1)= a\phi_X(r_1, r_1)=0,\]a \textit{contradiction.}
So, $\{r_1, s_1\}$ is a basis of $S_1$. Furthermore, we prove that
\[\begin{array}{ll} (i)\ S_1\cap T_1=0, & (ii)\ S_1+ T_1=
\mathcal{E}(X). \end{array}\] Indeed, $(i)$ since $\phi_X(r_1,
s_1)\neq 0$, we have $S_1\cap T_1=0.$ On the other hand, $(ii)$
for every $z\in \mathcal{E}(X)$, one has \[z= (-\phi_X(z, r_1)s_1+
\phi_X(z, s_1)r_1)+ (z+ \phi_X(z, r_1)s_1- \phi_X(z,
s_1)r_1),\]with \[-\phi_X(z, r_1)s_1+ \phi_X(z, s_1)r_1\in
S_1,\]and \[z+ \phi_X(z, r_1)s_1- \phi_X(z, s_1)r_1\in T_1.\]Thus,
\[\mathcal{E}(X)= S_1\oplus T_1.\]\textit{The restriction
}$\phi_1\equiv \phi_{1, X}$ of $\phi_X$ to $T_1$ \textit{is
non-degenerate,} because if $z_1\in T_1$ is such that $\phi_1(z_1,
z)=0$ for all $z\in T_1$, then $z_1\in T_1^\perp$ and hence
$z_1\in T_1\cap T_1^\perp= S^\perp_1\cap T^\perp_1= (S_1+
T_1)^\perp= \mathcal{E}(X)^\perp= 0,$ (the second equality derives
from Lemma \ref{lem3}); so $z_1=0$. $(T_1, \phi_1)$ \textit{is
thus a symplectic free $\mathcal{A}(X)$-module of rank $2(n-1)$.}
Repeating the construction above $n-1$ times, we obtain a strictly
decreasing sequence \[(\mathcal{E}(X), \phi_X)\supseteq (T_1,
\phi_1)\supseteq \cdots \supseteq (T_{n-1}, \phi_{n-1})\] of
symplectic free $\mathcal{A}(X)$-modules with rank $T_k= 2(n-k)$,
$k=1, \ldots, n-1,$ and also an increasing sequence \[\{r_1,
s_1\}\subseteq \{r_1, r_2; s_1, s_2\}\subseteq \cdots \subseteq
\{r_1, \ldots, r_n; s_1, \ldots, s_n\}\]of gauges; each satisfying
the relations (\ref{eq6}).

$(2)$ \underline{\textit{Case $I=J\neq \emptyset$.}} We may assume
without loss of generality that $I= J= \{1, 2, \ldots, k\},$ and
let $S$ be the subspace spanned by $\{r_1, \ldots, r_k; s_1,
\ldots, s_k\}$. Clearly, ${\phi_X}|_S$ is non-degenerate; by
Adkins-Weintraub~\cite[Lemma (2.31), p.360]{adkins}, it follows
that $S\cap S^\perp=0.$ On the other hand, let $z\in
\mathcal{E}(X)$. One has \[z= (-\sum^k_{i=1}\phi_X(z, r_i)s_i+
\sum^k_{i=1}\phi_X(z, s_i)r_i)+ (z+\sum^k_{i=1}\phi_X(z, r_i)s_i-
\sum^k_{i=1}\phi_X(z, s_i)r_i ),\]with \[-\sum^k_{i=1}\phi_X(z,
r_i)s_i+ \sum^k_{i=1}\phi_X(z, s_i)r_i\in S,\] and \[z+
\sum^k_{i=1}\phi_X(z, r_i)s_i- \sum^k_{i=1}\phi_X(z, s_i)r_i\in
S^\perp.\]Thus, \[\mathcal{E}(X)= S\oplus S^\perp.\]Based on the
hypothesis on $S_1$ the restriction ${\phi_X}|_S$ is a symplectic
$\mathcal{A}$-bilinear form. It is also easily seen that the
restriction ${\phi_X}_{S^\perp}$ is skew-symmetric. Moreover,
since ${S}\oplus S^\perp$ and $\mathcal{E}(X)^\perp=0$, if there
exist $z_1\in S^\perp$ such that $\phi_X(z_1, z)=0$ for all $z\in
S^\perp$, then $z_1\in \mathcal{E}(X)^\perp=0,$ i.e., $z_1=0.$
Thus, ${\phi_X}|_{S^\perp}$ is non-degenerate and hence a
symplectic $\mathcal{A}$-form. Applying Case $(1)$ , we obtain a
symplectic basis of $S^\perp$, which we denote as \[\{r_{k+1},
\ldots, r_n; s_{k+1}, \ldots, s_n\}.\]Then, \[\mathfrak{B}=\{r_1,
\ldots, r_n; s_1, \ldots, s_n\}\]is a symplectic basis of
$\mathcal{E}(X)$ with the required property.

$(3)$ \underline{\textit{Case $J\setminus I\neq \emptyset$ $($or
$I\setminus J\neq \emptyset)$.}} Suppose that $k\in J\setminus I;$
since $\phi_X$ is non-degenerate there exists $\overline{r}_k\in
\mathcal{E}(X)$ such that $\phi_X(\overline{r}_k, s_k)\neq 0$ in the
sense that $\phi_V({\overline{r}_k}|_V, {s_k}|_V)\neq 0$ for any
open $V\subseteq X$. In other words, the section $v\equiv
\phi_X(\overline{r}_k, s_k)\in \mathcal{A}(X)$ is nowhere zero, and
is therefore \textit{invertible} by virtue of the property of the
$\mathbb{C}$-algebra sheaf $\mathcal{A}$, as indicated in Scholium
\ref{scho2}. So, if $r_k:= v^{-1}\overline{r}_k,$ we have
$\phi_X(r_k, s_k)=1.$ Next, let us consider the
sub-$\mathcal{A}(X)$-module $R$, spanned by $r_k$ and $s_k$, viz.
$R= [r_k, s_k].$ As in Case $(1)$, we have \[\mathcal{E}(X)= R\oplus
R^\perp.\]Clearly, for every $i\in I$, $r_i\in R^\perp.$ To show
this, fix $i$ in $I$, and assume that $r_i= a_k+ bs_k+ x$, where $a,
b\in \mathcal{A}(X)$ and $x\in R^\perp.$ So, one has
\[\begin{array}{ll} 0= \phi_X(r_i, s_k)=a, & 0= \phi_X(r_i,
r_k)=b,\end{array}\]which corroborates the claim that $r_i\in
R^\perp$ for all $i\in I$. On the other hand, let us consider the
sub-$\mathcal{A}(X)$-module, $P$, generated by $B$. As in Case
$(2)$, one shows that \[\mathcal{E}(X)= P\oplus P^\perp.\]Since
$r_k\in \mathcal{E}(X)$, there exists $a_j\in \mathcal{A}(X)$ such
that \[r_k= \sum_{j\in J}a_js_j+ x,\]where $x\in P^\perp$. For any
$j\neq k$ in $J$, one has $\phi_X(r_k, s_j)=0.$ Thus, we have found
a section $r_k\in \mathcal{E}(X)$ such that $\phi_X(r_i, r_k)=0$ for
any $i\in I$ and $\phi_X(r_k, s_j)= \delta_{kj}$ for any $j\in J$.
Then $A\cup B\cup \{r_k\}$ is a family of linearly independent
sections: the equality \[a_kr_k+ \sum_{i\in I}a_ir_i+ \sum_{j\in
J}b_js_j=0\]implies that $a_k= a_i= b_j=0.$ Repeating this process
as many times as necessary, we are lead back to Case $(2)$, and the
proof is finished.
\end{proof}

Referring to Theorem \ref{theo3}, the basis $\mathfrak{B}$ is
called a \textbf{symplectic $\mathcal{A}(U)$-basis} of
$(\mathcal{E}(U), \phi_U)$.

\begin{corollary}\label{cor1}
If $(\mathcal{E}, \phi)$ is a symplectic free $\mathcal{A}$-module
of rank $2n$, then, for every open $U\subseteq X$,
\[\mathcal{E}(U)= H_1^U\oplus \cdots \oplus H_n^U,\]where $H^U_1,
\ldots, H_n^U$ are \textsf{pairwise orthogonal non-isotropic
two-dimensional sub-$\mathcal{A}(U)$-modules.}
\end{corollary}

\begin{proof}
The proof is similar \textit{to a good extent} to the first part
of the proof of Theorem \ref{theo3}. In fact, let $U$ be an open
subset of $X$ and $r_1\in \mathcal{E}(U)$, a \textit{nowhere-zero
section.} There exists a section $s_1$ in $\mathcal{E}(U)$ such
that $\phi_V({r_1}|_V, {s_1}|_V)\neq 0$ for any open $V\subseteq
U$. Clearly, $r_1, s_1$ must be linearly independent, and the
sub-$\mathcal{A}(U)$-module $H_1\equiv H_1^U:= [r_1, s_1],$
\textit{spanned} by $r_1$ and $s_1$, is \textit{non-isotropic.} As
in the proof of Theorem \ref{theo3}, Case $(1)$, one has
\[\mathcal{E}(U)= H_1\oplus H_1^\perp.\] The restriction
$\phi_{H_1^\perp}\equiv (\phi_U)|_{H_1^\perp}$ of $\phi_U$ to
$H_1^\perp$ is non-degenerate, because if $t\in H_1^\perp$ is such
that $\phi_{H_1^\perp}(t, z)=\phi_U(t, z)=0$ for all $z\in
H_1^\perp$, then $t\in H_1^{\perp\perp}\equiv (H_1^\perp)^\perp$
and hence $t\in H_1^\perp\cap H_1^{\perp\perp}= (H_1+
H_1^\perp)^\perp= \mathcal{E}(U)^\perp=0,$ which implies that
$t=0.$ Thus, $(H_1^\perp, \phi_{H_1^\perp})$ is a
\textit{symplectic free $\mathcal{A}(U)$-module of rank $2(n-1)$.}
Next, take a nowhere-zero $r_2\in H_1^\perp$; since $\phi_U(r_2,
r_1)= \phi_U(r_2, s_1)=0,$ there exists a section $s_2\in
H_1^\perp$ such that $\phi_V({r_2}|_V, {s_2}|_V)\neq 0$ for any
open $V\subseteq U$. As above, one has \[H^\perp_1= H_2\oplus
H_2^\perp,\]where $H_2:= [r_2, s_2].$ The \textit{direct
decomposition sum} of $\mathcal{E}(U)$ follows by repeating the
construction above $n-2$ times.
\end{proof}

Each sub-$\mathcal{A}(U)$-module $H_i^U$ in Corollary \ref{cor1}
has an ordered basis $(r_i, s_i)$ such that $(\phi_U(r_i,
s_i))|_V\equiv \phi_V({r_i}|_V, {s_i}|_V):= a_i|_V\neq 0$ for any
open subset $V$ of $U$. Then, based on the hypothesis that every
nowhere-zero section of $\mathcal{A}$ is invertible, see Scholium
\ref{scho2}, the restriction of $\phi_U$ to $H_i^U$ with respect
to the basis $(r_i, a_i^{-1}s_i)$ has matrix
\[\left(\begin{array}{rr}0 & 1\\ -1 & 0\end{array}\right).\]Hence,
we have

\begin{corollary}\label{cor2}
If $(\mathcal{E}, \phi)$ is a symplectic free $\mathcal{A}$-module
of rank $2n$, then for every open subset $U$ of $X$, there exists
an ordered basis of $\mathcal{E}(U)$ with respect to which
$\phi_U$ has matrix
\[
A^U_{2n}= \left(\begin{array}{ccccccc}\begin{tabular}{c|c}
$\begin{array}{rr} 0 & 1\\ -1 & 0\end{array}$ & \\ \hline & \\
\end{tabular}
& & \\
& \ddots & \\
& &
\begin{tabular}{c|c}
$\begin{array}{rr} 0 & 1\\ -1 & 0\end{array}$ & \\ \hline & \\
\end{tabular}
\end{array}
\right).
\]Moreover, symplectic $\mathcal{A}$-modules of the same rank are
isometric.
\end{corollary}

\section{Orthosymmetric $\mathcal{A}$-bilinear forms}

\begin{definition}
\emph{An $\mathcal{A}$-bilinear form $\phi: \mathcal{E}\oplus
\mathcal{E}\longrightarrow \mathcal{A}$ on an $\mathcal{A}$-module
$\mathcal{E}$ is called \textbf{orthosymmetric} if the following is
true: \begin{equation}\begin{array}{lll}\phi_U(r, s)=0 & \mbox{is
equivalent to} & \phi_U(s, r)=0,\end{array}\end{equation}for all $r,
s\in \mathcal{E}(U)$, with $U$ any open subset of $X$.}
\end{definition}

It is clear that if $\phi$ is orthosymmetric, then $\bot\equiv
\bot(\phi)= \top(\phi)\equiv \top,$ i.e. $\mathcal{F}^\perp=
\mathcal{F}^\top$ for any sub-$\mathcal{A}$-module $\mathcal{F}$ of
$\mathcal{E}$. Moreover, if $\phi$ is symmetric or skew-symmetric,
then $\phi$ is orthosymmetric. The following theorem shows that the
converse of the preceding statement is true on every open subset of
$X$.

\begin{theorem}\label{theo1}
Let $\mathcal{E}$ be an $\mathcal{A}$-module and $\phi\equiv
(\phi_U): \mathcal{E}\oplus \mathcal{E}\longrightarrow
\mathcal{A}$ an orthosymmetric $\mathcal{A}$-bilinear form. Then,
\textsf{componentwise} $\phi$ is \textsf{either symmetric or
skew-symmetric.}
\end{theorem}

\begin{proof}
Let $U$ be an open subset of $X$, and $r, s, t\in \mathcal{E}(U)$.
Clearly, we have \[\phi_U(r, \phi_U(r, t)s)- \phi_U(r, \phi_U(r,
s)t)= \phi_U(r, t)\phi_U(r, s)- \phi_U(r, s)\phi_U(r, t)=0,\]but
\[\phi_U(r, \phi_U(r, t)s- \phi_U(r, s)t)=0\]is equivalent to
\[\phi_U(\phi_U(r, t)s- \phi_U(r, s)t, r)=0;\]thus we obtain
\begin{equation}\label{eq1}
\phi_U(r, t)\phi_U(s, r)= \phi_U(r, s)\phi_U(t, r).
\end{equation}For $t=r$, $\phi_U(r, r)\phi_U(s, r)= \phi_U(r,
s)\phi_U(r, r).$ If \begin{equation}\begin{array}{ll}\phi_V(r|_V,
s|_V)\neq \phi_V(s|_V, r|_V), & \mbox{for any open $V\subseteq U$},
\end{array}\label{eq2}\end{equation}then \[\phi_U(r, r)=0.\](We note in
passing that (\ref{eq2}) suggests that both $\phi_V(r|_V, s|_V)$
and $\phi_V(s|_V, r|_V)$ are nowhere zero on $V$, because if, for
instance, $\phi_V(r|_V, s|_V)(x)=0$ for some $x\in V$ then
$\phi_V(r|_V, s|_V)=0$ on some open neighborhood $R\subseteq V$ of
$x$ (cf. Mallios~\cite[(3.7), p.13]{mallios}), i.e., assuming that
$(\rho^U_V)$ and $(\sigma^U_V)$ are the restriction maps for the
presheaves of sections of $\mathcal{E}$ and $\mathcal{A}$,
respectively, we have \[\sigma^U_R(\phi_U(s, r))=
\phi_R(\rho^U_R(s), \rho^U_R(r))\equiv \phi_R(s|_R,
r|_R)=0,\]which, by hypothesis, is equivalent to $\phi_R(r|_R,
s|_R)=0.$ That is a contradiction to (\ref{eq2}).)

Similarly, as \[\phi_U(s, \phi_U(s, t)r)- \phi_U(s, \phi_U(s,
r)t)=0,\]which, obviously, leads to
\begin{equation}\label{eq3}\phi_U(s, t)\phi_U(r, s)= \phi_U(s,
r)\phi_U(t, s),\end{equation}one has, for $t=s$, \[\phi_U(s,
s)\phi_U(r, s)= \phi_U(s, r)\phi_U(s, s).\]Using (\ref{eq2}), we
have \[\phi_U(s, s)=0.\]We actually have \textit{more} than just
what we have obtained so far. Indeed, if (\ref{eq2}) holds, then
$\phi_U(t, t)=0$ for all $t\in \mathcal{E}(U)$. We prove this
statement as follows.

(A) Let $\phi_V(r|_V, t|_V)\neq \phi_V(t|_V, r|_V)$ for any open
$V\subseteq U$. Since \begin{equation}\label{eq4}\phi_U(t,
r)\phi_U(s, t)= \phi_U(t, s)\phi_U(r, t),
\end{equation}by putting $s=t$, we have $\phi_U(t, t)=0.$

(B) Suppose that there exists an open $W\subseteq U$ such that
$\phi_W(r|_W, t|_W)= \phi_W(t|_W, r|_W).$ Then, by virtue of
(\ref{eq1}) and since $\phi_W(r|_W, s|_W)\neq \phi_W(s|_W, r|_W)$
everywhere on $W$, it follows that \[\phi_W(r|_W, t|_W)=0.\] On the
other hand, suppose that $\phi_V(s|_V, t|_V)\neq \phi_V(t|_V, s|_V)$
for any open $V\subseteq U$. Putting $r=t$ in (\ref{eq4}), one gets
$\phi_U(t, t)=0.$ Now, assume that there exists an open $T\subseteq
U$ such that $\phi_T(s|_T, t|_T)= \phi_U(t|_T, s|_T)$ and for any
open subset $V\subseteq U\setminus \overline{T},$ where
$\overline{T}$ is the closure of $T$ in $X$, $\phi_V(s|_V, t|_V)\neq
\phi_V(t|_V, s|_V).$ By virtue of (\ref{eq3}) and of \[\phi_T(s|_T,
r|_T)\neq \phi_T(r|_T, s|_T),\]it follows that \[\phi_T(s|_T, t|_T)=
\phi_T(t|_T, s|_T)=0.\]Hence,
\[\phi_T(r|_T+ t|_T, s|_T)= \phi_T(r|_T, s|_T)\neq \phi_T(s|_T,
r|_T)= \phi_T(s|_T, r|_T+ t|_T),\]and if we substitute $r|_T+ t|_T$
and $s|_T$ for $t|_V$ and $r|_V$ respectively in (A), we get
\[\phi_T(r|_T+ t|_T, r|_T+ t|_T)=0.\]But $\phi_T(r|_T,
r|_T)=0$ (since $\phi_U(r, r)=0$ and $T\subseteq U$ is open), then
if $\phi_T(r|_T, t|_T)= \phi_T(t|_T, r|_T)=0,$ one has
\begin{equation}\label{eq5}\phi_T(t|_T, t|_T)=0.\end{equation} If
$\phi_T(r|_T, t|_T)\neq 0\neq \phi_T(t|_T, r|_T)$ everywhere on $T$,
and $\phi_T(r|_T, t|_T)\neq \phi_T(t|_T, r|_T)$, we deduce from
(\ref{eq4}), by putting $s=t,$ $\phi_T(t|_T, t|_T)=0.$ If instead we
have $\phi_T(r|_T, t|_T)= \phi_T(t|_T, r|_T)$, we will end up with
\[\phi_T(r|_T, t|_T)=\phi_T(t|_T, r|_T)=0,\]which leads to
(\ref{eq5}) as previously shown. Next, $\phi_V(s|_V, t|_V)\neq
\phi_V(t|_V, s|_V)$ for every open $V\subseteq U\setminus
\overline{T}$, so $\phi_V(t|_V, t|_V)=0$ for every such $V$;
coupling the latter observation with (\ref{eq5}) and the fact that
sections are continuous, one gets in this case too that $\phi_U(t,
t)=0.$

We have shown that there are only two cases: either $\phi_U(r, r)=0$
for all $r\in \mathcal{E}(U),$ or for some $r\in \mathcal{E}(U)$,
$\phi_U(r, r)\neq 0,$ from which we deduce that $\phi_U(r, s)=
\phi_U(s, r)$ for all $r, s\in \mathcal{E}(U).$

Finally, we notice in ending the proof that if $\phi_U(r, r)=0$ for
all $r\in \mathcal{E}(U)$, then \[\phi_U(r, s)= -\phi_U(s, r)\]for
all $r, s\in \mathcal{E}(U)$.
\end{proof}

\begin{scholium}
\emph{In connection with the proof of Theorem \ref{theo1}, if there
exists an open subset $L\subseteq T$ such that $\phi_L(r|_L, t|_L)=
\phi_L(t|_L, r|_L)=0$ and $\phi_V(r|_V, t|_V)\neq \phi_V(t|_V,
r|_V)$ for every $V\subseteq T\setminus \overline{L},$ where
$\overline{L}$ is the closure of $L$ in $X$, then $\phi_L(t|_L,
t|_L)=0$ and $\phi_V(t|_V, t|_V)=0$ for every open $V\subseteq
T\setminus \overline{L}.$ Hence, $\phi_T(t|_T, t|_T)=0.$}
\end{scholium}

Referring still to Theorem \ref{theo1}, if $\phi_U$ is symmetric,
the geometry is called \textbf{orthogonal.} If $\phi_U$ is
skew-symmetric, the geometry is called \textbf{symplectic.} No other
case can occur if $\phi$ must be orthosymmetric. A \textit{pairing}
$(\mathcal{E}, \phi)$ is called \textit{symmetric} if every $\phi_U$
is symmetric, and \textit{skew-symmetric} if every $\phi_U$ is
skew-symmetric.

\begin{definition}
\emph{Let $(\mathcal{E}, \phi)\equiv [(\mathcal{E}, \phi);
\mathcal{A}]\equiv [((\mathcal{E}, \mathcal{E}); \phi);
\mathcal{A}]$ be a \textit{self-pairing} of an $\mathcal{A}$-module
$\mathcal{E}$, where $\phi$ is orthosymmetric. Then, by the
\textbf{radical} of $\mathcal{E}$, we mean the \textit{orthogonal}
$\mathcal{E}^\perp$. If $\mathcal{F}$ is a
\textit{sub-$\mathcal{A}$-module} of $\mathcal{E}$, the
\textit{radical,} $\mbox{rad}\ \mathcal{F}$, of $\mathcal{F}$ is
defined as $\mathcal{F}\cap \mathcal{F}^\perp.$ If rad
$\mathcal{F}=0$, $\mathcal{F}$ is said to be \textbf{non-isotropic;}
otherwise, it is called \textbf{isotropic.}}
\end{definition}

\begin{lemma}Let $(\mathcal{E}, \phi)$ be an $\mathcal{A}$-module
and $\mathcal{F}$ a sub-$\mathcal{A}$-module of $\mathcal{E}$. If
$\phi$ is orthosymmetric and $\mathcal{E}= \mathcal{F}\oplus
\mathcal{E}^\perp$, then $\mathcal{F}$ is non-isotropic.
\end{lemma}

\begin{proof}
Let $U$ be an open subset of $X$, and $r\in \mathcal{F}^\perp(U)$,
i.e. $\phi_V(\mathcal{F}(V), r|_V)=0$ for any open $V\subseteq U$.
But $\phi_V(\mathcal{E}^\perp(V), r|_V)=
\phi_V(\mathcal{E}^\top(V), r|_V)=0$ for any open $V\subseteq U$,
because $\mathcal{E}^\perp= \mathcal{E}^\top$, and therefore
\[\phi_V(\mathcal{F}(V)+ \mathcal{E}^\perp(V), r|_V)=
\phi_V(\mathcal{E}(V), r|_V)=0\]for any open $V\subseteq U$.
Hence, $r\in \mathcal{E}^\perp(U).$ We have thus  $
\mathcal{F}^\perp(U)\subseteq \mathcal{E}^\perp(U),$ so that
$\mathcal{F}(U)\cap \mathcal{F}^\perp(U)= (\mathcal{F}\cap
\mathcal{F}^\perp)(U):= (\mbox{rad}\mathcal{F})(U)=0.$
\end{proof}

\begin{definition}
\emph{Let $\mathcal{E}$ be an $\mathcal{A}$-module. An
$\mathcal{A}$-endomorphism $\phi\in \mbox{End}\ \mathcal{E}$ is
called \textbf{$\mathcal{A}$-involution} if $\phi^2=
\mbox{Id}_\mathcal{E}.$ An \textbf{$\mathcal{A}$-projection} is an
$\mathcal{A}$-endomorphism $p\in \mbox{End}\ \mathcal{E}$ such
that $p^2= p,$ in other words $p$ is \textit{idempotent}. The
$\mathcal{A}$-morphism $q\equiv \mbox{Id}_\mathcal{E}-p$ is
clearly an \textit{$\mathcal{A}$-projection;} $p$ and $q$ are
called \textbf{supplementary $\mathcal{A}$-projections.}}
\end{definition}

\begin{lemma}\label{lem1}
Let $(\mathcal{E}, \phi)$ be a free $\mathcal{A}$-module of finite
rank. Then, every non-isotropic free sub-$\mathcal{A}$-module
$\mathcal{F}$ of $\mathcal{E}$ is a direct summand of
$\mathcal{E}$; viz. \[\mathcal{E}= \mathcal{F}\bot\
\mathcal{F}^\perp.\]
\end{lemma}

\begin{proof}
Let us consider for any open subset $U\subseteq X$ a section $t\in
\mathcal{E}(U)$ and an $\mathcal{A}|_U$-\textit{form}
$\mathcal{F}|_U\longrightarrow \mathcal{A}|_U,$ defined as
follows: given any open $V\subseteq U$ and $s\in
\mathcal{F}|_U(V)= \mathcal{F}(V)$, one has \[s\longmapsto
\phi_V(t|_V, s).\] Since $\mathcal{F}$ is non-isotropic, the
restriction $\phi|_\mathcal{F}$ of $\phi$ on $\mathcal{F}$ is
non-degenerate; consequently the above $\mathcal{A}|_U$-form may
be represented by a unique element (in fact, a section)
$p_U(t)\equiv p(t)\in \mathcal{F}(U)\cong \mathcal{F}^\ast(U)$ in
such a way that \[\phi_V(t|_V, s)= (\phi|_\mathcal{F})_V(p(t)|_V,
s)= \phi_V(p(t)|_V, s)\]for all $s\in \mathcal{F}(V)$. For $r,
t\in \mathcal{E}(U),$ we have \[\phi_V((r+t)|_V, s)=
\phi_V(p(r+t)|_V, s),\]and on the other hand
\[\begin{array}{lll}\phi_V((r+t)|_V, s) & = & \phi_V(r|_V, s)+
\phi_V(t|_V, s)\\ & = & \phi_V(p(r)|_V, s)+ \phi_V(p(t)|_V, s)\\ &
= & \phi_V(p(r)|_V+ p(t)|_V, s),\end{array}\]for all $s\in
\mathcal{F}(V)$ and where $V$ is open in $U$. But for every $t\in
\mathcal{E}(U),$ $p(t)$ is unique, therefore $p(r+t)= p(r)+ p(t).$
Likewise, one shows that for all $\alpha\in \mathcal{A}(U)$,
$p(\alpha t)=\alpha p(t).$ The observation undertaken about $p$
means that $p: \mathcal{E}(U)\longrightarrow \mathcal{E}(U)$ is
$\mathcal{A}(U)$-linear. Next, since $p^2=p$, then the
$\mathcal{A}(U)$-morphism $p: \mathcal{E}(U)\longrightarrow
\mathcal{F}(U)$ is an $\mathcal{A}(U)$-projection. Furthermore,
since \[\phi_V((t- p(t))|_V, s)= \phi_V(t|_V- p(t)|_V, s)=0\] for
all $t\in \mathcal{E}(U)$ and $s\in \mathcal{F}(V),$ with $V$ open
in $U$, the supplementary $\mathcal{A}(U)$-projection $q:= I-p$ is
such that for all $t\in \mathcal{E}(U)$, $q(t)\equiv (I-p)(t)\in
\mathcal{F}^\perp(U),$ i.e. $q$ maps $\mathcal{E}(U)$ on
$\mathcal{F}^\perp(U).$ Hence, every element $t\in
\mathcal{E}(U)$, where $U$ runs over the open subsets of $X$, may
be written as \[t= p(t)+ (t-p(t))\]with $p(t)\in \mathcal{F}(U)$
and $t-p(t)\in \mathcal{F}^\perp(U)$, thus \[\mathcal{E}(U)=
\mathcal{F}(U)\oplus \mathcal{F}^\perp(U)= (\mathcal{F}\oplus
\mathcal{F}^\perp)(U)\]within $\mathcal{A}(U)$-isomorphisms (see
cf. Mallios~\cite[relation (3.14), p.122]{mallios} for the
$\mathcal{A}(U)$-isomorphism $\mathcal{F}(U)\oplus
\mathcal{F}^\perp(U)= (\mathcal{F}\oplus \mathcal{F}^\perp)(U)$).
Finally, since $\mathcal{F}$ is non-isotropic, it follows that
\[\mathcal{E}(U)= (\mathcal{F}\bot\ \mathcal{F}^\perp)(U)\]for
every open $U\subseteq X$. Thus, we reach the sought
$\mathcal{A}$-isomorphism of the lemma.
\end{proof}

\begin{definition}
\emph{A \textbf{convenient $\mathcal{A}$-module} is a
\textit{self-pairing} $(\mathcal{E}, \phi)$, where $\mathcal{E}$ is
a \textit{free $\mathcal{A}$-module of finite rank} and $\phi$ an
\textit{orthosymmetric $\mathcal{A}$-bilinear form,} such that the
following conditions are satisfied. \begin{enumerate} \item
[{$(1)$}] If $\mathcal{F}$ is a \textit{free
sub-$\mathcal{A}$-module} of $\mathcal{E}$, then the
\textit{orthogonal} $\mathcal{F}^\perp$ and the \textit{radical} rad
$\mathcal{F}$ are \textit{free sub-$\mathcal{A}$-modules} of
$\mathcal{E}$.
\item [{$(2)$}] Every \textit{free sub-$\mathcal{A}$-module} $\mathcal{F}$ of
$\mathcal{E}$ is \textit{orthogonally reflexive,} i.e.
$(\mathcal{F}^\perp)^\perp\equiv \mathcal{F}^{\perp\perp}=
\mathcal{F}.$ \item [{$(3)$}] The \textit{intersection of any two
free sub-$\mathcal{A}$-modules} of $\mathcal{E}$ is a \textit{free
sub-$\mathcal{A}$-module.} \end{enumerate}}
\end{definition}

\begin{lemma}
If $(\mathcal{E}, \phi)$ is a convenient $\mathcal{A}$-module, then,
given any two free sub-$\mathcal{A}$-modules $\mathcal{G}$ and
$\mathcal{H}$ of $\mathcal{E}$, one has \[(\mathcal{G}\cap
\mathcal{H})^\perp= \mathcal{G}^\perp+ \mathcal{H}^\perp.\]
\end{lemma}

\begin{proof}
By virtue of Lemma \ref{lem3}, we have \[\begin{array}{lll}
(\mathcal{G}^\perp+ \mathcal{H}^\perp)^\perp & = &
(\mathcal{G}^\perp)^\perp\cap (\mathcal{H}^\perp)^\perp \\ & = &
\mathcal{G}\cap \mathcal{H},\ \mbox{since $\mathcal{E}$ is
convenient,}\end{array}\]whence \[\mathcal{G}^\perp+
\mathcal{H}^\perp= (\mathcal{G}^\perp+
\mathcal{H}^\perp)^{\perp\perp}= (\mathcal{G}\cap
\mathcal{H})^\perp.\]
\end{proof}

\begin{lemma}
If $(\mathcal{E}, \phi)$ is a convenient $\mathcal{A}$-module and
$\mathcal{F}$ a \textsf{non-isotropic free sub-$\mathcal{A}$-module}
of $\mathcal{E}$, then $(\mathcal{F}, \widetilde{\phi})$, where
$\widetilde{\phi}:= \phi|_\mathcal{F},$ is a convenient
$\mathcal{A}$-module.
\end{lemma}

\begin{proof}
Let $\bot(\widetilde{\phi})$ and $\bot(\phi)$ denote orthogonality
with respect to $\widetilde{\phi}$ and $\phi$ respectively. Let
$\mathcal{G}$ and $\mathcal{H}$ be sub-$\mathcal{A}$-modules of
$\mathcal{F}$.

$(1)$ That $\mathcal{G}^{\perp(\widetilde{\phi})}$ and
rad$_{\widetilde{\phi}}\mathcal{G}$ are free
sub-$\mathcal{A}$-modules is clear. Indeed,
\[\mathcal{G}^{\perp(\widetilde{\phi})}=
\mathcal{G}^{\perp(\phi)}\cap \mathcal{F}\] and
\[\mbox{rad}
_{\widetilde{\phi}}\mathcal{G}:= \mathcal{G}\cap
\mathcal{G}^{\perp(\widetilde{\phi})}= \mathcal{G}\cap
(\mathcal{G}^{\perp(\phi)}\cap \mathcal{F})= (\mathcal{G}\cap
\mathcal{G}^{\perp(\phi)})\cap \mathcal{F}=:
\mbox{rad}_\phi\mathcal{G}\cap \mathcal{F}.\]

$(2)$ By an easy calculation, we have \[\begin{array}{lll}
\mathcal{G}^{\perp(\widetilde{\phi})\perp(\widetilde{\phi})} & = &
(\mathcal{G}^{\perp(\widetilde{\phi})})^{\perp(\phi)}\cap
\mathcal{F}\\ & = & (\mathcal{G}^{\perp(\phi)}\cap
\mathcal{F})^{\perp(\phi)}\cap \mathcal{F}\\ & = &
(\mathcal{G}^{\perp(\phi)\perp(\phi)}+
\mathcal{F}^{\perp(\phi)})\cap \mathcal{F}\\ & = & (\mathcal{G}\cap
\mathcal{F})+ (\mathcal{F}^{\perp(\phi)}\cap \mathcal{F})\\ & = &
\mathcal{G}\cap \mathcal{F}\\ & = & \mathcal{G}\end{array}\]

$(3)$ Immediate.
\end{proof}

We now turn to the following theorem.

\begin{theorem}\label{theo4}
Let $(\mathcal{E}, \phi)$ be a \textsf{non-isotropic
skew-symmetric convenient $\mathcal{A}$-module,} and $\mathcal{F}$
a totally isotropic sub-$\mathcal{A}$-module of rank $k$. Then,
there is a \textsf{non-isotropic sub-$\mathcal{A}$-module}
$\mathcal{H}$ of $\mathcal{E}$ of the form \[\mathcal{H}=
\mathcal{H}_1\bot\cdots \bot \mathcal{H}_k,\]where if
$\mathcal{F}(U)= [r_{1, U}, \cdots, r_{k, U}]$ with $U$ an open
subset of $X$, then $r_{i, U}\in \mathcal{H}_i(U)$ for $1\leq
i\leq k$.
\end{theorem}

\begin{proof}
Suppose that $k=1$, i.e. $\mathcal{F}\cong \mathcal{A}$. If
$\mathcal{F}(X)= [r_X]$ with $r_X\in \mathcal{E}(X)$ a
nowhere-zero section, then for every open $U\subseteq X$,
$\mathcal{F}(U)= [r_U]$, where $r_U= {r_X}|_U.$ Since $\phi_X$ is
non-degenerate, there exists a nowhere-zero section $s_X\in
\mathcal{E}(X)$ such that $\phi_U({r_X}|_U, {s_X}|_U)\neq 0$ for
every open $U\subseteq X$. The correspondence \[U\longmapsto
\mathcal{H}(U):= [r_U, s_U]\equiv [{r_X}|_U, {s_X}|_U],\]where $U$
runs over the open sets in $X$, along with the obvious restriction
maps, yields a complete presheaf of $\mathcal{A}$-modules on $X$.
Clearly, the pair $(\mathcal{H}, \widetilde{\phi}),$ where
$\widetilde{\phi}$ is the $\mathcal{A}$-bilinear morphism
$\widetilde{\phi}: \mathcal{H}\oplus \mathcal{H}\longrightarrow
\mathcal{A}$ such that \[(r_U, s_U)\longmapsto
\widetilde{\phi}_U(r_U, s_U):= \phi_U(r_U, s_U),\]is
non-isotropic. Hence, the theorem holds for the case $k=1$. Let us
now proceed by induction to $k>1$. To this end, put
$\mathcal{F}_{k-1}\cong \mathcal{A}^{k-1}$ and $\mathcal{F}_k:=
\mathcal{F}\cong \mathcal{A}^k.$ Then,
$\mathcal{F}_{k-1}\varsubsetneqq \mathcal{F}_k,$ so
$\mathcal{F}_k^\perp\varsubsetneqq \mathcal{F}^\perp_{k-1}.$ Since
orthogonal of free sub-$\mathcal{A}$-modules in a convenient
$\mathcal{A}$-module are free sub-$\mathcal{A}$-modules, the
inclusion $\mathcal{F}_k^\perp\varsubsetneqq
\mathcal{F}^\perp_{k-1}$ implies that, if
$\mathcal{F}^\perp_{k-1}\cong \mathcal{A}^m$ and
$\mathcal{F}_k^\perp\cong \mathcal{A}^n$ with $n<m$, then
$\mathcal{F}^\perp_{k-1}\setminus \mathcal{F}^\perp_k\cong
\mathcal{A}^{m-n}.$ For every open $U\subseteq X$, pick $s_{k,
U}\in \mathcal{F}^\perp_{k-1}(U)\setminus \mathcal{F}^\perp_k(U),$
and put $\mathcal{H}_k(U)= [r_{k, U}, s_{k, U}].$ The
correspondence \[U\longmapsto \mathcal{H}_k(U),\]where $U$ is open
in $X$, along with the obvious restriction maps, is a complete
presheaf of $\mathcal{A}(U)$-modules. Since $\phi_U(r_{i, U},
s_{k, U})=0$ for $1\leq i\leq k-1,$ $\phi_U(r_{k, U}, s_{k,
U})\neq 0$. Hence, $\mathcal{H}_k(U)$ is a non-isotropic
$\mathcal{A}(U)$-plane containing $r_{k, U}.$ By Lemma \ref{lem1}
$\mathcal{E}= \mathcal{H}_k\bot \mathcal{H}_k^\perp.$ Since $r_{k,
U}, s_{k, U}\in \mathcal{F}^\perp_{k-1}(U),$
$\mathcal{H}_k(U)\subseteq \mathcal{F}_{k-1}^\perp(U)$ for every
open $U\subseteq X$; so $\mathcal{H}_k\subseteq
\mathcal{F}^\perp_{k-1},$ which in turn implies that
$\mathcal{F}_{k-1}\subseteq \mathcal{H}^\perp_k.$ Apply an
inductive argument to $\mathcal{F}_{k-1}$ regarded as a
sub-$\mathcal{A}$-module of the non-isotropic skew-symmetric
convenient $\mathcal{A}$-module $\mathcal{H}^\perp_k$.
\end{proof}

We are now set for the analog of the Witt's theorem; to this end
we assume that $(X, \mathcal{A})$ \textit{is an algebraized space
satisfying the condition of Scholium \ref{scho2}.} For the
classical Witt's theorem, see Adkins-Weintraub~\cite[pp
368-387]{adkins}, Artin~\cite[pp 121, 122]{artin}, Berndt~\cite[p
21]{berndt}, Crumeyrolle~\cite[pp 11, 12]{crumeyrolle},
Deheuvels~\cite[pp 148, 152]{deheuvels}, Lang~\cite[pp 591,
592]{lang}, O'Meara~\cite[p 9]{omeara}.

\begin{theorem}~$(${\bf Witt's Theorem}$)$~ Let $\mathcal{E}\equiv (\mathcal{E}, \phi)$ and
$\mathcal{E}'\equiv (\mathcal{E}', \phi')$ be \textsf{isometric
non-isotropic skew-symmetric convenient $\mathcal{A}$-modules,}
$\mathcal{F}\equiv (\mathcal{F}, \widetilde{\phi})$, where
$\widetilde{\phi}:= \phi|_{\mathcal{F}}$, a \textsf{free
sub-$\mathcal{A}$-module} of $\mathcal{E}$, and $\sigma\equiv
(\sigma_U): \mathcal{F}\longrightarrow \mathcal{E}'$ an
\textsf{$\mathcal{A}$-isometry of $\mathcal{F}$ into
$\mathcal{E}'$.} Then, $\sigma$ \textsf{extends to an
$\mathcal{A}$-isometry of $\mathcal{E}$ onto $\mathcal{E}'$.}
\end{theorem}

\begin{proof}
 Since $\mathcal{E}$ is convenient and $\mathcal{F}$ is a free
 sub-$\mathcal{A}$-module of $\mathcal{E}$, there exists a free
 sub-$\mathcal{A}$-module of $\mathcal{E}$ such that
 $\mathcal{F}= \mathcal{G}\bot~
\mbox{rad}~\mathcal{F},$ where if $\mathcal{F}$ and rad
$\mathcal{F}$ are $\mathcal{A}$-isomorphic to $\mathcal{A}^k$ and
$\mathcal{A}^l$ respectively, then $\mathcal{G}$ is
$\mathcal{A}$-isomorphic to $\mathcal{A}^{k-l}$. By Lemma
\ref{lem3}(1), $\mathcal{F}^\perp\subseteq \mathcal{G}^\perp;$ since
$\mathcal{G}^\perp$ is non-isotropic and skew-symmetric, and rad
$\mathcal{F}$ is a totally isotropic free sub-$\mathcal{A}$-module,
by applying Theorem \ref{theo4}, we see that there is a free
sub-$\mathcal{A}$-module $\mathcal{H}$ of $\mathcal{G}^\perp$ of the
form \[\mathcal{H}:= \mathcal{H}_1\bot\cdots \bot \mathcal{H}_l\]in
which each $\mathcal{H}_i$ is a non-isotropic free
sub-$\mathcal{A}$-module of rank $2$ and such that if
\[(\mbox{rad}~\mathcal{F})(U)= [r_{1, U}, \cdots, r_{l, U}],\]where
$U$ is an open subset of $X$, then $r_{i, U}\in \mathcal{H}_i(U)$
with $i= 1, \ldots, l$. Since $\mathcal{H}$ is non-isotropic it
splits $\mathcal{G}^\perp$: $\mathcal{G}^\perp= \mathcal{H}\bot~
\mathcal{J};$ in fact, $\mathcal{J}\cong \mathcal{H}^\perp$ (see
Lemma \ref{lem1}). Hence, \[\mathcal{E}= \mathcal{G}^\perp\bot~
\mathcal{G}= \mathcal{H}\bot~\mathcal{J}\bot~\mathcal{G},\]within
$\mathcal{A}$-isomorphisms respectively. Put $\mathcal{F}':=
\sigma(\mathcal{F}),$ $\mathcal{G}':= \sigma(\mathcal{G})$ and
$r'_{i, U}:= \sigma_U(r_{i, U}),$ $1\leq i\leq l,$ for every open
$U\subseteq X$. Now, let us fix $U$ in the topology of $X$.
Clearly,
\[\mathcal{F}'(U)= \{t\in \mathcal{E}'(U):~\phi'_U(\sigma_U(s),
t)=0,~ s\in \mathcal{F}(U)\}\]and \[\mathcal{F}(U)^\perp=\{z\in
\mathcal{E}(U):~\phi_U(s, z)=0, ~s\in \mathcal{F}(U)\}.\]For every
$z\in \mathcal{F}(U)^\perp,$ we have for all $s\in \mathcal{F}(U)$
\[\phi'_U(\sigma_U(s), \sigma_U(z))= \phi_U(s, z)=0;\]we thus deduce
that \[\sigma_U(\mathcal{F}^\perp(U))=
\sigma_U(\mathcal{F}(U)^\perp)\subseteq \mathcal{F}'(U)^\perp=
{\mathcal{F}'}^\perp(U).\]hence, \[\begin{array}{lll}
\sigma_U(\mbox{rad}~\mathcal{F}(U)) & := &
\sigma_U(\mathcal{F}(U)\cap \mathcal{F}(U)^\perp) \\ & = &
\sigma_U(\mathcal{F}(U))\cap \sigma_U(\mathcal{F}(U)^\perp),
~\mbox{since $\sigma_U$ is an $\mathcal{A}(U)$-isomorphism}\\ &
\subseteq & \mathcal{F}'(U)\cap \mathcal{F}'(U)^\perp=
\mbox{rad}~\mathcal{F}'(U):=
\mbox{rad}~\sigma_U(\mathcal{F}(U)).\end{array}\]Conversely, let
$t\mbox{rad}~\sigma_U(\mathcal{F}(U)):= \sigma_U(\mathcal{F}(U))\cap
\sigma_U(\mathcal{F}(U))^\perp.$ As $\sigma_U$ is an
$\mathcal{A}(U)$-isomorphism there exists a unique $s\in
\mathcal{F}(U)$ such that $t=\sigma_U(s).$ But
\[0=\phi'_U(\sigma_U(r), \sigma_U(s))= \phi_U(r, s)\]for every $r\in
\mathcal{F}(U).$ Consequently, $s\in \mathcal{F}(U)^\perp.$ Thus,
\[s\in \mathcal{F}(U)\cap \mathcal{F}(U)^\perp=:
\mbox{rad}~\mathcal{F}(U);\]hence
\[t\in \sigma_U(\mbox{rad}~\mathcal{F}(U)),\]from which we deduce that
\[\mbox{rad}~\sigma_U(\mathcal{F}(U))\subseteq
\sigma_U(\mbox{rad}~\mathcal{F}(U)).\]The end result of this
argument is that \[\mbox{rad}~\sigma_U(\mathcal{F}(U))=
\sigma_U(\mbox{rad}~\mathcal{F}(U)).\]Since $U$ is arbitrary, it
follows that \[\mbox{rad}~\mathcal{F}'\equiv
\mbox{rad}~\sigma(\mathcal{F})= \sigma(\mbox{rad}~\mathcal{F})\cong
\mathcal{A}^l.\] Since $\sigma$ is an $\mathcal{A}$-isomtery, we
obtain that \[\mathcal{F}':= \sigma(\mathcal{F})=
\sigma(\mathcal{G}\bot~\mbox{rad}~\mathcal{F})=
\mathcal{G}'\bot~\mbox{rad}~\mathcal{F}'\]is a radical splitting of
$\mathcal{F}'$. Repeating the early argument, we have
\[\mathcal{E}'= \mathcal{H}'\bot ~\mathcal{J}'\bot ~\mathcal{G}'\]in
which \[\mathcal{H}'= \mathcal{H}_1'\bot\cdots \bot \mathcal{H}'_l\]
with each $\mathcal{H}_i'$ a non-isotropic free
sub-$\mathcal{A}$-module of rank $2$ such that if
\[(\mbox{rad}~\mathcal{F}')(U)= [r'_{1, U}, \cdots, r'_{l,
U}],\]where $U$ is open in $X$, then $r'_{i, U}\in
\mathcal{H}'_i(U)$ for every $1\leq i\leq l.$ Suppose for every
$i=1, \ldots, l$, $\mathcal{H}_i(U)= [r_{i, U}, s_{i, U}]$ and
$\mathcal{H}_i'(U)= [r'_{i, U}, s'_{i, U}].$ Let $\alpha=
(\alpha_U): \mathcal{H}\longrightarrow \mathcal{H}'$ be an
$\mathcal{A}$-morphism, given by the prescription
\[\begin{array}{lll} \alpha_U(r_{i, U})= r'_{i, U} & \mbox{and} &
\alpha_U(s_{i, U})= s'_{i, U}\end{array}\]for every open $U\subseteq
X$ and $i=1, \ldots, l.$ That $\alpha$ is an
$\mathcal{A}$-isomorphism is clear. Next, observe that for every
open $U\subseteq X$ and $i=1, \ldots, l$, since $\phi_U$ and
$\phi'_U$ are non-degenerate, $\phi_U(r_{i, U}, s_{i, U})$ and
$\phi'_U(r'_{i, U}, s'_{i, U})$ are nowhere zero sections;
consequently based on the hypothesis regarding the coefficient
algebra sheaf $\mathcal{A}$, $\phi_U(r_{i, U}, s_{i, U})$ and
$\phi'_U(r'_{i, U}, s'_{i, U})$ are invertible. It is clear that for
every open $U\subseteq X$ and $i=1, \ldots, l$, \[\mathcal{H}_i'(U)=
[r'_{i, U}, s'_{i, U}\phi_U(r_{i, U}, s_{i, U})(\phi'_U(r'_{i, U},
s'_{i, U}))^{-1}].\]The $\mathcal{A}$-morphism $\beta\equiv
(\beta_U): \mathcal{H}\longrightarrow \mathcal{H}'$ given by
\[\begin{array}{lll} \beta_U(r_{i, U})= r'_{i, U} & \mbox{and} &
\beta_U(s_{i, U})= s'_{i, U}\phi_U(r_{i, U}, s_{i,
U})(\phi'_U(r'_{i, U}, s'_{i, U}))^{-1}\end{array}\] is clearly an
$\mathcal{A}$-isomorphism such that \[\phi'_U(\beta_U(r_{i, U}),
\beta_U(s_{i, U}))= \phi_U(r_{i, U}, s_{i, U});\]in other words,
$\beta$ is an $\mathcal{A}$-isometry of $\mathcal{H}$ onto
$\mathcal{H}'$. Furthermore, $\beta$ agrees with $\sigma$ on each
$r_{i, U}$, and hence on rad $\mathcal{F}$. Also, the given $\sigma$
carries $\mathcal{G}$ onto $\mathcal{G}'$ isomorphically. Hence
$\sigma$ extends to an $\mathcal{A}$-isometry of $\mathcal{H}\bot
~\mathcal{G}$ onto $\mathcal{H}'\bot~\mathcal{G}'$. Now, rank
$(\mathcal{E})= \mbox{rank}~(\mathcal{E}')$; hence rank
$(\mathcal{J})= \mbox{rank}~(\mathcal{J}');$ hence by Corollary
\ref{cor2} there is an $\mathcal{A}$-isometry of $\mathcal{J}$ onto
$\mathcal{J}'$. Hence, finally, $\sigma$ extends to an isometry of
$\mathcal{E}= (\mathcal{H}\bot~\mathcal{G})\bot ~\mathcal{J}$ onto
$\mathcal{E}'= (\mathcal{H}'\bot~\mathcal{G}')\bot~\mathcal{J}'$.
\end{proof}

\noindent Patrice P. Ntumba\\{Department of Mathematics and
Applied Mathematics}\\{University of Pretoria}\\ {Hatfield 0002,
Republic of South Africa}\\{Email: patrice.ntumba@up.ac.za}

\noindent Adaeze Orioha\\{Department of Mathematics and Applied
Mathematics}\\{University of Pretoria}\\ {Hatfield 0002, Republic
of South Africa}\\{Email: s27632832@tuks.co.za}

\end{document}